\documentclass[a4paper,12pt,reqno,intlimits]{extarticle}
\usepackage[cp1251]{inputenc}
\usepackage[T2A]{fontenc}
\usepackage[english,russian]{babel}
\usepackage{amsfonts,amsmath,amsxtra,amsthm,amssymb,latexsym,mathrsfs,graphicx,mathtools,cite,nicefrac,url,multirow}
\usepackage{algorithm}
\usepackage{algorithmic}
\floatname{algorithm}{Алгоритм}

\usepackage[multiple]{footmisc}

\theoremstyle{plain}
\newtheorem{theorem}{Теорема}
\newtheorem{lemma}{Лемма}

\newtheorem{corollary}{Следствие}

\theoremstyle{definition}

\numberwithin{theorem}{section}
\numberwithin{lemma}{section}
\numberwithin{proposition}{section}
\numberwithin{corollary}{section}
\numberwithin{remark}{section}
\numberwithin{definition}{section}
\numberwithin{example}{section}

\newcommand\br[1]{\left(#1\right)}

\newcommand\abr[1]{\left\langle#1\right\rangle}

\DeclareMathOperator*{\argmin}{arg\,min}

\author{Ф.С. Стонякин}

\title{Адаптивные зеркальные спуски для задач выпуклого программирования с использованием $\delta$-субградиентов}
\date{}

\begin{document}

\maketitle

\setcounter{section}{1}

Задачи негладкой оптимизации с ограничениями часто возникают в приложениях и им посвящены всё новые работы, в том числе в области субградиентых методов (см. например \cite{Dvur_Nurminsky} и имеющиеся там ссылки). Отметим, что во многих задачах операция нахождения точного вектора субградиента целевого функционала может быть достаточно затратной. Одним из подходов к указанной проблеме может быть использование вместо субградиентов $f$ {\it $\delta$-субградиентов} $\nabla_{\delta} f(x)$ в текущей точке $x$:
$$f(y)-f(x)\geqslant\langle\nabla_{\delta} f(x),y-x\rangle-\delta$$
для всякого $y\in Q$ при некотором фиксированном $\delta\geq0$ (ясно, что при $\delta=0$ $\nabla_{\delta} f(x)=\nabla f(x)$~--- обычный субградиент). Например, для
$$f(x)=\max_{y\in P}\varphi(x,y)$$
выпуклой функции $\varphi$ по $x$ и непрерывной по $y\in P$ ($P$~--- некоторый компакт) $\nabla_{\delta} f(x)$ есть субградиент по $x$ функционала $\varphi(x,y_{\delta})$, где
$$\max_{y\in P}\varphi(x,y)-\varphi(x,y_{\delta})\leqslant\delta.$$

Оказывается, что в рассматриваемых нами методах зеркального спуска (\cite{Stonyakin_Monograpf}, глава 5) можно заменить обычный субградиент $\nabla f(x)$ на $\delta$-субградиент $\nabla_{\delta} f(x)$. Оценки качества решения при этом изменятся на величину $O(\delta)$. В настоящем тексте, который дополняет \cite{Stonyakin_Monograpf}, мы покажем, как это возможно сделать. Через $||\cdot||$ будем обозначать норму в рассматриваемом конечномерном пространстве переменных задачи, а под $||\cdot||_*$  будем понимать норму в сопряжённом (двойственном) пространстве к исходному.

Начнём с важного вспомогательного утверждения, которое есть аналог для $\delta$-субградиентов известного утверждения из работ Ю.\,Е.\,Нестерова (см. \cite{Nesterov_2010}, лемма 3.2.1). Оно позволяет получать оценки скорости сходимости для широкого класса целевых функционалов. Пусть при $x$ и $y\in Q$
$$v^{\delta}_{f}(x,y)=\biggl\langle\frac{\nabla_{\delta} f(x)}{\|\nabla_{\delta} f(x)\|_{*}}, x-y\biggr\rangle$$
при $\nabla_{\delta} f(x)\neq 0$ и $v^{\delta}_{f}(x,y)=0$ при $\nabla_{\delta} f(x)= 0$.

\begin{lemma}\label{lem-neserov-delta}
Если $f: Q\rightarrow \mathbb{R}$ выпуклый функционал и при $\tau>0$
$$\omega(\tau)=\max_{x\in Q}\{f(x)-f(x_{*})\,|\,\|x-x_{*}\|\leqslant\tau\},\;\omega(\tau)= 0\text{ при }\tau \leqslant 0.$$
Тогда для всякого $x\in Q$
$$f(x)-f(x_{*})\leqslant\omega(v^{\delta}_{f}(x,x_{*}))+\delta.$$
\end{lemma}
\begin{proof}

1)Покажем, что при $v^{\delta}_{f}(x,x_{*}) > 0$
$$v^{\delta}_{f}(x,x_{*})=\min_{y}\{\|y-x_{*}\|\,|\,\langle \nabla_{\delta} f(x),y-x\rangle=0\}.$$
Выберем $y_{*}$ так, чтобы $\|y_{*}-x_{*}\|=\displaystyle\min_{y}\{\|y-x_{*}\|\,|\,\langle \nabla_{\delta} f(x),y-x\rangle=0\}$.
Далее, для фиксированного $x$ верно $\langle \nabla_{\delta} f(x),y_{*}-x\rangle=0$. Поэтому существует такое $\lambda>0$, что $\nabla_{\delta} f(x)=\lambda s$ при $\langle s, y_{*}-x_{*}\rangle=\|y_{*}-x_{*}\|$ и $\|s\|_{*}=1$. Имеем:
$$0=\langle \nabla_{\delta} f(x),y_{*}-x\rangle=\lambda \langle s,y_{*}-x_{*}\rangle+\langle \nabla_{\delta} f(x),x_{*}-x\rangle,$$
т.е.
$$\lambda=\frac{\langle \nabla_{\delta} f(x),x-x_{*}\rangle}{\|y_{*}-x_{*}\|}=\|\nabla_{\delta} f(x)\|_{*}.$$
Это и означает, что $v^{\delta}_{f}(x,x_{*})=\|y_{*}-x_{*}\|$.

2) Поскольку $\langle \nabla_{\delta} f(x),y-x\rangle=0$, то $f(y)-f(x)\geqslant -\delta$ для всякого $y$, при котором $\langle\nabla_{\delta} f(x),y-x\rangle=0$. Далее,
$$f(x)-f(x_{*})\leqslant f(y)-f(x_{*})+\delta\leqslant \omega(v^{\delta}_{f}(x,x_{*}))+\delta,$$
что и требовалось. Случай $v^{\delta}_{f}(x,x_{*}) \leqslant 0$ очевиден.
\end{proof}

Приведём также следующий аналог базовой леммы для зеркальных спусков для $\delta$-субградиентов $\nabla_{\delta} f$ \cite{Dvur_Nurminsky}. Аналогично \cite{Dvur_Nurminsky} через $V(y, x) = d(y) - d(x) - \langle \nabla d(x), y - x \rangle $ мы обозначаем расхождение или дивергенцию Брэгмана, порождённую некоторой 1-сильно выпуклой относительно нормы рассматриваемого пространства $||\cdot||$ функцией расстояния $d$. 
\begin{lemma}
\label{p0_Lm:MDPropDelta}
Пусть $f$~--- некоторая выпуклая функция на $Q$, $h > 0$~--- размер шага. Пусть точка $y$ определяется формулой $ y = \mathrm{Mirr}[x](h \cdot (\nabla_{\delta} f(x)))$. Тогда для всякого $z \in Q$
$$h \cdot \big( f(x) - f(z)\big) \leqslant h \cdot \langle \nabla_{\delta}f(x), x - z \rangle + h \cdot \delta\leqslant$$
$$\leqslant \frac{h^2}{2} \| \nabla_{\delta} f (x) \|_*^2 + h \cdot \delta + V(z,x) - V(z,y).$$
\end{lemma}

Обозначим через $\varepsilon>0$ фиксированную точность, $x^0$~--- начальное приближение такое, что для некоторого $\Theta_0>0$ верно неравенство $V(x_*,x^0)\leqslant \Theta_0^2$. Будем предполагать далее, что $||\nabla_{\delta} g(x)||_{*} \leqslant M_g$ ($M_g > 0$) при всех $x \in Q$ \footnote{В частности, если $\delta$-субградиент вводится на всё пространстве, то для этого достаточно допустить лишь $|g(x)-g(y)|\leqslant M_g||x-y||\;\forall x,y\in \mathbb{R}^n$.}.

Рассмотрим следующий алгоритм \ref{alg4}.
\begin{algorithm}
\caption{Адаптивный зеркальный спуск (аналог алгоритма 19 \cite{Stonyakin_Monograpf})}
\label{alg4}
\begin{algorithmic}[1]
\REQUIRE $\varepsilon>0,\;\Theta_0:d(x_*)\leqslant\Theta_0^2$
\STATE $x^0=argmin_{x\in Q}\,d(x)$
\STATE $I=:\emptyset$
\STATE $k\leftarrow0$
\REPEAT
  \IF{$g(x^k)\leqslant\varepsilon||\nabla_{\delta}g(x^k)||_* + \delta$}
    \STATE $h_k\leftarrow\frac{\varepsilon}{||\nabla_{\delta} f(x^k)||_{*}^2}$
    \STATE $x^{k+1}\leftarrow Mirr_{x^k}(h_k\nabla_{\delta} f(x^k))\;\text{// \emph{"продуктивные шаги"}}$
    \STATE $k\rightarrow I$
  \ELSE
    \STATE $h_k\leftarrow\frac{\varepsilon}{||\nabla_{\delta} g(x^k)||_{*}}$
    \STATE $x^{k+1}\leftarrow Mirr_{x^k}(h_k\nabla_{\delta} g(x^k))\;\text{// \emph{"непродуктивные шаги"}}$
  \ENDIF
  \STATE $k\leftarrow k+1$
\UNTIL{
  \begin{equation}\label{eq1}
  \frac{2\Theta_0^2}{\varepsilon^2} \leqslant\sum\limits_{k\in I}\frac{1}{||\nabla_{\delta}f(x^k)||_{*}^2}+|J|,
  \end{equation}
  }
где $|J|$~--- количество непродуктивных шагов (мы обозначим через $|I|$ количество продуктивных шагов, то есть $|I|+|J|=N$).
\ENSURE{$\widehat{x}=\frac{1}{\sum\limits_{k\in I}h_k}\sum\limits_{k\in I}h_kx^k$.}
\end{algorithmic}
\end{algorithm}

На базе леммы \ref{p0_Lm:MDPropDelta} можно вывести следующий результат об оценке качества найденного решения предложенного метода. При этом отметим, что в аналогичных \eqref{p5_eq3} неравенствах для непродуктивных итераций сокращаются величины, содержащие $\delta$, ввиду того, что $\delta$ содержится в неравенствах для проверки продуктивности итераций. Это позволяет сохранить критерий остановки алгоритма~\ref{alg4}, аналогичный критерию такого типа для алгоритма~\ref{alg4}.

\begin{theorem}\label{p5_th1_new_methods}
После остановки предложенного алгоритма~\ref{alg4} для
$$\widehat{x}=\frac{1}{\sum_{k\in I}h_k}\sum\limits_{k\in I}h_kx^k$$
верно $f(\widehat{x})-f(x_*)\leqslant \varepsilon$ и $g(\widehat{x})\leqslant\varepsilon M_g$.
\end{theorem}
\begin{proof}
1) Если $k\in I$, то
\begin{equation}\label{p5_eq2a}
\begin{split}
h_k\br{f(x^k)-f(x_*)}\leqslant h_k\abr{\nabla_{\delta} f(x^k),x^k-x_*}+h_k\delta\leqslant\\
\leqslant\frac{h_k^2}{2}||\nabla_{\delta} f(x^k)||_*^2+V(x_*,x^k)-V(x_*,x^{k+1})+h_k\delta=\\
=\frac{\varepsilon^2}{2}\cdot\frac{1}{||\nabla_{\delta} f(x^k)||_*^2}+V(x_*,x^k)-V(x_*,x^{k+1})+h_k\delta.
\end{split}
\end{equation}

2) Если $k\not\in I$, то $$\frac{g(x^k)}{||\nabla_{\delta} g(x^k)||_*}>\varepsilon+\frac{\delta}{||\nabla_{\delta} g(x^k)||_*}$$ и $$\frac{g(x^k)-g(x_*)}{||\nabla g(x^k)||_*}\geqslant\frac{g(x^k)}{||\nabla g(x^k)||_*}>\varepsilon+\frac{\delta}{||\nabla_{\delta} g(x^k)||_*}=\varepsilon+\frac{h_k\delta}{\varepsilon}.$$
Поэтому верны неравенства
\begin{equation}\label{p5_eq3}
\begin{split}
\varepsilon^2+h_k\delta<h_k\br{g(x^k)-g(x_*)}+h_k\delta\leqslant\frac{h_k^2}{2}||\nabla g(x^k)||_*^2+\\
+V(x_*,x^k)-V(x_*,x^{k+1})+h_k\delta=\frac{\varepsilon^2}{2}+V(x_*,x^k)-V(x_*,x^{k+1})+h_k\delta,\text{ или}\\
\frac{\varepsilon^2}{2}<V(x_*,x^k)-V(x_*,x^{k+1}).
\end{split}
\end{equation}

3) После суммирования неравенств~\eqref{p5_eq2a} и~\eqref{p5_eq3} имеем:
$$\sum_{k\in I}h_k\br{f(x^k)-f(x_*)}\leqslant$$
$$\leqslant\sum_{k\in I}\frac{\varepsilon^2}{2||\nabla_{\delta} f(x^k)||_*^2}-\frac{\varepsilon^2|J|}{2}+V(x_*,x^0)-V(x_*,x^{k+1})+\sum_{k\in I}h_k\delta=$$
$$=\frac{\varepsilon}{2}\sum_{k\in I}h_k-\frac{\varepsilon^2|J|}{2}+\Theta_0^2+\sum_{k\in I}h_k\delta=$$
$$=\varepsilon\sum_{k\in I}h_k-\frac{\varepsilon^2}{2}\br{\sum_{k\in I}\frac{1}{||\nabla f(x^k)||_*^2}+|J|}+\Theta_0^2+\sum_{k\in I}h_k\delta.$$

После выполнения критерия остановки алгоритма~\ref{alg4} будет верно
$$\sum_{k\in I}h_k\br{f(x^k)-f(x_*)}\leqslant\varepsilon\sum_{k\in I}h_k+\sum_{k\in I}h_k\delta,$$
откуда для $\widehat{x}:=\sum\limits_{k\in I}\frac{h_kx^k}{\sum_{k\in I}h_k}$
$$f(\widehat{x})-f(x_*)\leqslant\varepsilon+\delta.$$
При этом $\forall k\in I\;g(x^k)\leqslant\varepsilon||\nabla g(x^k)||_*\leqslant\varepsilon M_g$, откуда
$$g(\widehat{x})\leqslant\frac{1}{\sum\limits_{k\in I}h_k}\sum_{k\in I}h_kg(x^k)\leqslant\varepsilon M_g.$$

Остается лишь показать, что множество продуктивных шагов $I$ непусто. Если $I=\emptyset$, то $|J|=N$ и п.~14 листинга означает, что $N\geqslant\frac{2\Theta_0^2}{\varepsilon^2}$. С другой стороны, из~\eqref{p5_eq3}:
$$\frac{\varepsilon^2N}{2}<V(x_*,x^0)\leqslant\Theta_0^2,$$
то есть получили противоречие и $I\neq\emptyset$. Теорема доказана.
\end{proof}


\begin{algorithm}[H]
	\caption{Адаптивный зеркальный спуск (аналог алгоритма 17 \cite{Stonyakin_Monograpf})}
	\label{p5_algorithm11}
	\begin{algorithmic}[1]
		\REQUIRE $ \text{точность} \ \varepsilon>0; \text{начальная точка} \ x^0; \Theta_0; Q; d(\cdot). $
		\STATE$I:=\emptyset$
		\STATE$N\leftarrow 0$
		\REPEAT
		\IF{$g(x^N) \leqslant \varepsilon+\delta$}
		\STATE $h_N\leftarrow\frac{\varepsilon}{\lVert \nabla_{\delta}f(x^N) \rVert_{*}}$
		\STATE$x^{N+1}\leftarrow Mirr_{x^N}(h_N\nabla_{\delta}f(x^N)) \;\; \text{("продуктивные шаги")}$
		\STATE $N\rightarrow I$
		\ELSE
		\STATE $h_N\leftarrow\frac{\varepsilon}{\lVert \nabla_{\delta}g(x^N) \rVert_{*}^2}$
		\STATE $x^{N+1}\leftarrow Mirr_{x^N}(h_N\nabla_{\delta}g(x^N)) \;\; \text{("непродуктивные шаги")}$
		\ENDIF
		\STATE $N\leftarrow N+1$
		\UNTIL $\Theta_0^2 \leqslant \frac{\varepsilon^2}{2}\left(|I|+\sum\limits_{k\not\in I}\frac{1}{\lVert \nabla_{\delta} g(x^k) \rVert_{*}^2}\right)$
		\ENSURE $\bar{x}^N := \argmin\limits_{x^k, k\in I} f(x^k)$
	\end{algorithmic}
\end{algorithm}

\begin{algorithm}
\caption{Адаптивный зеркальный спуск, постоянное количество итераций (аналог алгоритма 20 \cite{Stonyakin_Monograpf}).}
\label{p5_alg41}
\begin{algorithmic}[1]
\REQUIRE $\varepsilon>0,\Theta_0: \,d(x_*)\leqslant\Theta_0^2$
\STATE $x^0=argmin_{x\in Q}\,d(x)$
\STATE $I=:\emptyset$
\STATE $N\leftarrow0$
\REPEAT
    \IF{$g(x^N)\leqslant\varepsilon\|\nabla_{\delta}g(x^N)\|_*+\delta$}
        \STATE $M_N=\|\nabla_{\delta} f(x^N)\|_*$, $h_N=\frac{\varepsilon}{M_N}$
        \STATE $x^{N+1}=Mirr_{x^N}(h_N\nabla_{\delta} f(x^N))\text{ // \emph{"продуктивные шаги"}}$
        \STATE $N\rightarrow I$
    \ELSE
        \STATE $M_N=\|\nabla_{\delta} g(x^N)\|_*$, $h_N=\frac{\varepsilon}{M_N}$
        \STATE $x^{N+1}=Mirr_{x^N}(h_N\nabla_{\delta} g(x^N))\text{ // \emph{"непродуктивные шаги"}}$
    \ENDIF
    \STATE $N\leftarrow N+1$
\UNTIL{$2\frac{\Theta_0^2}{\varepsilon^2}\leqslant N$}
\ENSURE $\bar{x}^N:=argmin_{x^k,\;k\in I}\,f(x^k)$
\end{algorithmic}
\end{algorithm}

На базе лемм \ref{lem-neserov-delta} (для продуктивных шагов) и \ref{p0_Lm:MDPropDelta} (для непродуктивных шагов) можно вывести следующие результаты об оценке качества найденного решения предложенных методов. При этом отметим, что в аналогичных \eqref{p5_eq44} неравенствах для непродуктивных итераций сокращаются величины, содержащие $\delta$, ввиду того, что $\delta$ содержится в неравенствах для проверки продуктивности итераций. Это позволяет предложить критерии остановки алгоритмов~\ref{p5_algorithm11} и \ref{p5_alg41}, аналогичные критериям такого типа для алгоритмов~~\ref{p5_algorithm11} и \ref{p5_alg41}.

\begin{theorem}\label{p5_th_est}
Пусть известна константа $ \Theta_ {0}> 0 $ такая, что $d(x_*) \leqslant \Theta_{0}^{2}$. Если $\varepsilon > 0$~--- фиксированное число, то алгоритм~\ref{p5_algorithm11} работает не более
\begin{equation}\label{p5_atsenka_alg1}
N=\left\lceil\frac{2\max\{1, M_g^2\}\Theta_0^2}{\varepsilon^2}\right\rceil
\end{equation}
итераций, причём после его остановки справедливо неравенство
	\begin{equation}\label{p5_eq09a}
	\min\limits_{k \in I} v_f^{\delta}(x^k,x_*)\leqslant\varepsilon,\; \max\limits_{k \in I}g(x^k)\leqslant\varepsilon.
	\end{equation}
	\label{p5_theorem1}
\end{theorem}
\begin{proof}
Пусть $[N]=\{k\in\overline{0,N-1}\}$, $J=[N]\setminus I$~--- набор номеров непродуктивных шагов.

1) Для продуктивных шагов по лемме~\ref{lem-neserov-delta} имеем, что $$h_k\langle\nabla_{\delta} f(x^k), x^k-x\rangle\leqslant\frac{h_k^2}{2}||\nabla_{\delta} f(x^k)||_*^2+V(x,x^k)-V(x,x^{k+1}).$$

Примем во внимание, что $\frac{h_k^2}{2}||\nabla_{\delta} f(x^k)||_*^2=\frac{\varepsilon^2}{2}$. Тогда
\begin{equation}\label{p5_eq001}
\begin{split}
h_k\langle\nabla_{\delta} f(x^k),x^k-x_*\rangle=\varepsilon\left\langle\frac{\nabla_{\delta} f(x^k)}{||\nabla_{\delta} f(x^k)||_*},\, x^k-x_*\right\rangle=\varepsilon v_f^{\delta}(x^k,x_*) \leqslant\\
\leqslant \frac{\varepsilon^2}{2} + V(x_*,x^k)-V(x_*,x^{k+1}).
\end{split}
\end{equation}

2) Аналогично для «непродуктивных» шагов $k\in J$ (под $g(x^k)$ мы понимаем любое ограничение, для которого $g(x^k)>\varepsilon$) по лемме~\ref{p0_Lm:MDPropDelta}:

\begin{equation}\label{p5_eq44}
\begin{split}
h_k\left(g(x^k)-g(x_*)\right)+h_k\delta\leqslant\frac{h_k^2}{2}||\nabla g(x^k)||_*^2+V(x_*,x^k)-V(x_*,x^{k+1})+h_k\delta=\\
=\frac{\varepsilon^2}{2||\nabla g(x^k)||_*^2}+V(x_*,x^k)-V(x_*,x^{k+1})+h_k\delta.
\end{split}
\end{equation}

3) Из~\eqref{p5_eq001} и~\eqref{p5_eq44} при $x=x_*$ имеем:
\begin{equation}\label{p5_eq10}
\begin{split}
\varepsilon\cdot \sum\limits_{k\in I}v_f^{\delta}(x^k,x_*)+\sum\limits_{k\in J}\frac{\varepsilon^2 (g(x^k)-g(x_*))}{2||\nabla g(x^k)||_*^2}\leqslant\\
\leqslant\frac{\varepsilon^2}{2} \cdot |I| +\sum\limits_{k=0}^{N-1}(V(x_*,x^k)-V(x_*,x^{k+1}))+\sum\limits_{k\in J}h_k\delta.
\end{split}
\end{equation}

Заметим, что для любого $k \in J$
\begin{equation}\label{p5_eqiavgmk}
g(x^k)-g(x_*)\geqslant g(x^k)>\varepsilon+\delta.
\end{equation}
С учетом
$$\sum\limits_{k=0}^{N-1}(V(x_*,x^k)-V(x_*,x^{k+1}))\leqslant\Theta_0^2$$
неравенство~\eqref{p5_eq10} может быть преобразовано следующим образом:
$$\varepsilon\sum\limits_{k\in I}v_f^{\delta}(x^k,x_*)\leqslant |I|\cdot\frac{\varepsilon^2}{2}+\Theta_0^2-\sum\limits_{k \in J}\frac{\varepsilon^2}{2||\nabla_{\delta} g(x^k)||_*^2},$$
$$\sum\limits_{k\in I}v_f^{\delta}(x^k,x_*)\geqslant|I|\min\limits_{k\in I}v_f(x^k,x_*).$$

Таким образом,
$$\varepsilon \cdot \min\limits_{k \in I} v_f^{\delta}(x^k,x_*)\cdot |I| \leqslant \frac{\varepsilon^2}{2}\cdot|I|+\Theta_0^2-\sum\limits_{k \in J}\frac{\varepsilon^2}{2||\nabla_{\delta} g(x^k)||_*^2}\leqslant \varepsilon^2\cdot|I|,$$
откуда
\begin{equation}\label{p5_eq12}
\min\limits_{k \in I} v_f^{\delta}(x^k,x_*) \leqslant \varepsilon.
\end{equation}

В завершении покажем, что $|I|\neq 0$. Предположим обратное: $|I|=0\Rightarrow|J|=N$, т.е. все шаги непродуктивны. Тогда с учётом~\eqref{p5_eqiavgmk} получаем, что
$$
h_k(g(x^k)-g(x_*)) > \frac{\varepsilon^2}{\|\nabla_{\delta} g(x^k)\|_*^2}+\frac{\varepsilon\delta}{\|\nabla_{\delta} g(x^k)\|_*^2}
$$
и
$$\sum\limits_{k=0}^{N-1} h_k(g(x^k)-g(x_*))\leqslant\sum\limits_{k=0}^{N-1}\frac{\varepsilon^2}{2\|\nabla_{\delta} g(x^k)\|_*^2}+\Theta_0^2+\sum_{k=0}^{N-1}h_k\delta\leqslant$$
$$\leqslant\sum\limits_{k=0}^{N-1}\frac{\varepsilon^2}{\|\nabla_{\delta} g(x^k)\|_*^2}+\sum_{k=0}^{N-1}\frac{\varepsilon\delta}{\|\nabla_{\delta} g(x^k)\|_*^2}.$$

Итак, получили противоречие. Это означает, что $|I|\neq 0$.

Отметим, что в силу предположения для функционального ограничения на любой итерации работы алгоритма~\ref{p5_algorithm11} справедливо неравенство $||\nabla_{\delta} g(x^k)||_* \leqslant  M_g$. Поэтому критерий остановки алгоритма~\ref{p5_algorithm11} будет заведомо выполнен не более, чем после~\eqref{p5_atsenka_alg1} итераций.
\end{proof}



По схеме рассуждений из (\cite{StonStep2020}, Theorem 2) по аналогии c предыдущим доказательством можно проверить следующий результат.

\begin{theorem}
Пусть $\varepsilon > 0$~--- фиксированное число и выполнен критерий остановки алгоритма~\ref{p5_alg41}. Тогда
\begin{equation}\label{p5_eq1_new_methods}
\min\limits_{k \in I} v_f^{\delta}(x^k,x_*) \leqslant \varepsilon ,\; \max\limits_{k \in I}g(x^k)\leqslant M_g\varepsilon + \delta.
\end{equation}
\end{theorem}

\begin{corollary}
	Предположим, что $f(x) = \max\limits_{i = \overline{1, m}} f_i(x)$, где $f_i$ дифференцируемы в каждой точке $x \in Q$ и
	$$	||\nabla f_i(x)-\nabla f_i(y)||_*\leqslant L_i||x-y|| \quad \forall x,y\in Q.	$$
	Тогда после $$N=\left\lceil\frac{2\max\{1, M_g^2\}\Theta_0^2}{\varepsilon^2}\right\rceil$$ шагов работы алгоритма~\ref{p5_algorithm11} для некоторого $\nabla_{\delta}f(x_*)$ \footnote{можно рассматривать и обычный субградиент, т.е. $\delta = 0$} будет верно следующее неравенство:
		$$\min\limits_{k\in I}f(x^k)-f(x_*)\leqslant \|\nabla_{\delta}f(x_*)\|_*\varepsilon+\frac{L}{2}\varepsilon^2+\delta,$$
	где $L = \max\limits_{i = \overline{1, m}} L_i$.
\end{corollary}

Аналогичный результат можно сформулировать и для алгоритма 3.

Настоящий текст есть дополнение с уточнениями к заключительным замечаниям главы 5 монографии \cite{Stonyakin_Monograpf}.

\end{document}